\documentclass[a4paper, reqno, 11pt]{amsart}

\usepackage{amscd, amssymb, amsthm, amsmath, color}
\usepackage{graphicx, psfrag, enumerate}

\theoremstyle{plain}
\newtheorem{thm}{Theorem}[section]
\newtheorem{cor}[thm]{Corollary}
\newtheorem{lem}[thm]{Lemma}
\newtheorem{prop}[thm]{Proposition}

\newtheorem{prob}[thm]{Problem}

\theoremstyle{remark}

\newcommand{\stat}{\mathsf{stat}}
\newcommand{\Stat}{\mathsf{Stat}}

\newcommand{\cyc}{\mathsf{cyc}}
\newcommand{\inv}{\mathsf{inv}}
\newcommand{\wex}{\mathsf{wex}}
\newcommand{\exc}{\mathsf{exc}}
\newcommand{\rlmin}{\mathsf{rlmin}}
\newcommand{\des}{\mathsf{des}}

\newcommand{\fix}{\mathsf{fix}}

\newcommand{\Fix}{\mathsf{Fix}}
\newcommand{\Sfix}{\mathsf{Sfix}}
\newcommand{\sfix}{\mathsf{sfix}}
\newcommand{\Cyc}{\mathsf{Cyc}}
\newcommand{\Scyc}{\mathsf{Scyc}}
\newcommand{\scyc}{\mathsf{scyc}}
\newcommand{\Excl}{\mathsf{Excl}}
\newcommand{\Rlmip}{\mathsf{Rlmip}}

\begin{document}
\title{Statistics of Partial Permutations via Catalan matrices}

\author{Yen-Jen Cheng}
\address{Department of Applied Mathematics, National Yang Ming Chiao Tung University, Hsinchu 30093, Taiwan, ROC}
\email[Yen-Jen Cheng]{yjc7755@gmail.com, yjc7755@nycu.edu.tw }

\author{Sen-Peng Eu}
\address{Department of Mathematics, National Taiwan Normal University, Taipei 11677, Taiwan, ROC}
\email[Sen-Peng Eu]{speu@math.ntnu.edu.tw}

\author{Hsiang-Chun Hsu}
\address{Department of Mathematics, Tamkang University, Taipei 251301, Taiwan, ROC}
\email[Hsiang-Chun Hsu]{hchsu0222@gmail.com}

\subjclass[2010]{05A05, 05A19}

\keywords{permutation, partial permutation, Catalan matrix, statistic, set-valued statistic, fixed point, inversion, descent, cycle, excedance, right-to-left minimum, connected permutation, cycle-up-down permutation}

\thanks{Y.-J. Cheng is partially supported by MOST 110-2811-M-A49-505, 
S.-P. Eu is partial supported by MOST 110-2115-M-003-011-MY3 and H.-C. Hsu is partially supported by MOST 110-2115-M-032-004-MY2}

\date{\today}

\maketitle


\begin{abstract}
A generalized Catalan matrix $(a_{n,k})_{n,k\ge 0}$ is generated by two seed sequences
$\mathbf{s}=(s_0,s_1,\ldots)$ and $\mathbf{t}=(t_1,t_2,\ldots)$ 
together with a recurrence relation. By taking $s_\ell=2\ell+1$ and $t_\ell=\ell^2$ we can interpret $a_{n,k}$ as the number of  partial permutations, which are $n\times n$ $0,1$-matrices of $k$ zero rows with at most one $1$ in each row or column. In this paper we prove that most of fundamental statistics and some set-valued statistics on permutations can also be defined on partial permutations and be encoded in the seed sequences. Results on two interesting permutation families, namely the connected permutations and cycle-up-down permutations, are also given.
\end{abstract}


\section{Introduction}
\subsection{Catalan matrices and Partial permutations}
Given two sequences $\mathbf{s}=(s_0,s_1,\ldots)$ and $\mathbf{t}=(t_1,t_2,\ldots)$ with $t_i\ne 0$ for all $i$, 
we can define an infinite lower triangular matrix $A^{\mathbf{s},\mathbf{t}}=(a_{n,k})_{n,k\geq 0}$ from the recurrence relation
$$
\begin{cases}
a_{0,0} = 1, a_{0,k}=0 \quad (k> 0), \\
a_{n,k} = a_{n-1,k-1}+s_ka_{n-1,k}+t_{k+1}a_{n-1,k+1}\quad(n\geq 1),
\end{cases}
$$
where $a_{n, -1}$ is defined to be zero. 
$A^{\mathbf{s},\mathbf{t}}$ is called the {\it generalized Catalan matrix} associated with the \emph{seed sequences} (or \emph{seeds} for short) $\mathbf{s}$ and $\mathbf{t}$. 
The numbers $a_{n,0}$ in the left-most column are called the {\it generalized Catalan numbers} associated with  $\mathbf{s}$ and $\mathbf{t}$. 
These names are from the fact that if we take $\mathbf{s}=(0,0,0,\ldots)$ and $\mathbf{t}=(1,1,1,\ldots)$, then $a_{2n+1,0}=0$ and $a_{2n,0}=\frac{1}{n+1}{2n\choose n}$ is the $n$-th ordinary Catalan number. In this paper when we refer to Catalan numbers or Catalan matrices we always mean the generalized ones.  Catalan matrices are closely related with orthogonal polynomials, lattice paths and Riordan groups, see~\cite{Aigner_07, Corteel_Kim_Stanton_16, Sadjang_Koepf_Foupouagnigni_15} for more information.

\medskip
Let  $\mathfrak{S}_{n}$ be the set of permutations of length $n$.
A permutation $\pi\in \mathfrak{S}_n$ can be seen as an $n\times n$ permutation matrix, which is a $0,1$-matrix with exactly one 1 in each column and each row.
Define a {\it partial permutation} of order $n$ to be an $n\times n$ $(0,1)$-matrix with at most one 1 in each column and each row. 
Let  $\mathfrak{P}_{n,k}$ be the set of partial permutations of order (length) $n$ with $k$ zero rows (hence $k$ zero columns), 
and $\mathfrak{P}_n=\cup_k \mathfrak{P}_{n,k}$.
It is easy to see that $$|\mathfrak{P}_{n,k}|={n\choose k}^2(n-k)!.$$ 
By definition $\mathfrak{P}_{n,0}$ is exactly the set of the ordinary permutation matrices.
Our study is motivated by the observation that if we take $(s_\ell, t_\ell) = (2\ell+1, \ell^2)$ as the seeds, then 
 in the associated Catalan matrix we have $a_{n,k}= |\mathfrak{P}_{n,k}|$ and $a_{n,0}=n!$.

The notion of partial permutations is a natural extension of ordinary permutations. They have been studied in various settings such as percolation model on $\mathbb{Z}^2$~\cite{Derycke_18}, Erd\"os-Ko-Rado theorem~\cite{Ku_Leader_06}, cross-intersecting families~\cite{Borg_10}, and braid monoid~\cite{East_07}. Note that there is a different (nonequivalent) combinatorial structure bear the same name~\cite{Claesson_11, Ivanov_01} and is also widely studied. Readers should not confuse between these two.

\subsection{Theme of this paper}

The central theme of this paper is that most of the classic statistics over ordinary permutations can be also defined on partial permutations and it is delightful to us that all the information is encoded in the seeds. 

We take Corollary 3.3 as an introductory example. Let 
$$
[n] :=1+q+q^2+\dots +q^{n-1}, \qquad [n]! =[1][2]\dots [n], \qquad {n\brack k} :=\frac{[n]!}{[k]![n-k]!}
$$
be the $q$-analogs of the number $n$, the factorial $n!$ and the binomial coefficient ${n\choose k}$ respectively.
There we will prove that if we set the seeds to be 
$$(s_\ell,t_\ell)=(2[\ell]_q q^{\ell}+q^{2\ell}, [\ell]_q^2 q^{2\ell-1}),$$ then 
$a_{n,k}$ is exactly the generating function of $\mathfrak{P}_{n,k}$ with respect to the inversion statistic $\mathsf{inv}$ (to be defined later).
Note that in~\cite{Derycke_18} another `inversion' statistic is defined but it is different from ours. Moreover, we can have the explicit form
$$a_{n,k}=\sum_{\pi\in \mathfrak{P}_{n,k}}{q^{\inv(\pi)}}={n\brack k}^2[n-k]!.$$
When $k=0$ we recover the inversion enumerator of ordinary permutations.

\medskip
The rest of the paper is organized as follows. Notations and the definitions of statistics will be put in Section 2. In Section 3 we discuss
the joint distribution of inversion, weak excedance and right-to-left minimum. In Section 4 we discuss the relation among descent, weak-excedance and excedance. 
In Section 5 we introduce the concept of extended Catalan matrix which help us
to generalize some ordinary integer-valued statistics to set-valued statistics and investigate the joint distribution of set-valued fixed point, cycle and right-to-left-minimum over partial permutations. In Sections 6 and 7 we discuss two interesting permutation families whose statistics can also be encoded in the seeds, namely the connected partial permutations and cycle-up-down partial permutations. 


\section{Statistics over partial permutations}
\subsection{Partial permutations}
A permutation in $\mathfrak{S}_n$ is isomorphic to an $n\times n$ $(0,1)$-matrix with exactly $1$ in each column and each row. A {\it partial permutation} $\pi$ of order (or length) $n$ is defined to be an $n\times n$ $(0,1)$-matrix with \emph{at most} one 1 in each column and each row and we call this matrix the \emph{matrix notation} of $\pi$ and let $\pi_{i,j}$ be its $(i,j)$-entry. Let $\mathfrak{P}_n$ be the set of partial permutations of order $n$ and 
$\mathfrak{P}_{n,k}$ be the set of partial permutations of order $n$ with $k$ zero rows (hence $k$ zero columns). 
For example,
$$\pi=
\begin{bmatrix}
0 & 0 & 0 & 0 & 0 & 1 & 0 & 0 & 0\\ 
1 & 0 & 0 & 0 & 0 & 0 & 0 & 0 & 0\\ 
0 & 0 & 0 & 0 & 0 & 0 & 0 & 0 & 0\\ 
0 & 0 & 0 & 1 & 0 & 0 & 0 & 0 & 0\\ 
0 & 0 & 0 & 0 & 0 & 0 & 0 & 0 & 1\\ 
0 & 1 & 0 & 0 & 0 & 0 & 0 & 0 & 0\\ 
0 & 0 & 1 & 0 & 0 & 0 & 0 & 0 & 0\\ 
0 & 0 & 0 & 0 & 0 & 0 & 0 & 0 & 0\\ 
0 & 0 & 0 & 0 & 1 & 0 & 0 & 0 & 0\\ 
\end{bmatrix}
$$
is a partial permutation in $\mathfrak{P}_{9,2}$.

The \emph{two-line notation} of a partial permutation $\pi\in \mathfrak{P}_n$ is a $2\times n$ array, in which
the entries of the first row are $1,2,\dots ,n$ and the $i$-th entry of the second row is $j$ if 
the $(i,j)$-entry in its matrix notation is $1$, or $\mathsf{X}$ if the $i$-th row is a zero row. 
By omitting the first row we obtain the \emph{one-line notation} $\pi=\pi_1\pi_2\dots \pi_n$.
For the above $\pi$ its two-line notation and one-line notation are respectively
$$\pi=\left(
\begin{array}{ccccccccc}
1 & 2 & 3 & 4 & 5 & 6 & 7 & 8 & 9 \\
6 & 1 & \mathsf{X} & 4 & 9 & 2 & 3 & \mathsf{X} & 5
\end{array}
\right) \qquad \mbox{and} \qquad \pi=61\mathsf{X}4923\mathsf{X}5.$$

Regard the two line notation as a mapping sending $i$ to $\pi_i$ for $1\le i \le n$.
Choose an $i$ and successively trace the images and pre-images we obtain either a cycle starting from and back to $i$, which we call a \emph{full cycle}, or a sequence stops at $\mathsf{X}$, which we call a \emph{partial cycle} if we write this sequence as a cycle. 
In this way $\pi$ is decomposed into a collection of full and partial cycles, which is called the \emph{cycle notation} of $\pi$. For example,
the above $\pi$ has the cycle notation $$\pi=(216)(73\mathsf{X})(4)(8\mathsf{X})(59).$$ Here $(216),(4),(59)$ are full cycles and $(73\mathsf{X}), (8\mathsf{X})$ are partial cycles. Clearly for any $\pi\in \mathfrak{P}_{n,k}$ there are $k$ partial cycles.


\subsection{Statistics of a partial permutation}
We extend the definitions of classic statistics of a permutation to a partial permutation. 
Let $\pi=\pi_1\pi_2\dots \pi_n$ be a partial permutation. 

\begin{itemize}
\item $\mathsf{fix}$. The \emph{fixed point} statistic is defined by 
$$\mathsf{fix}(\pi):= \text{the number of $i$ such that $\pi_i=i$}.$$

\item $\mathsf{cyc}$. The \emph{cycle} statistic of $\pi$ is defined by
$$\mathsf{cyc}(\pi):= \text{number of full cycles.}$$ 
We will also see
$\mathsf{cyc}_{\ge 2}(\pi)$ later, which is the number of full cycles of length greater than or equal to $2$.
Note that for an ordinary permutation all cycles are full.

\item $\mathsf{inv}$.  For an ordinary permutation $\pi$, the definition of the  \emph{inversion} statistic is 
$$\mathsf{inv}(\pi):=|\{(i,j): i<j \text{ and } \pi_i>\pi_j\}|.$$
We need the following equivalent fact.. If we write the ordinary permutation $\pi$ in the matrix notation and delete all entries to the right and below of each $1$ then $\mathsf{inv}(\pi)$ is the remaining number of $0$. 
For a partial partition $\pi$ we do similarly: for each $1$ we delete all entries to its right and below and define
$$\mathsf{inv}(\pi):= \text{number of remaining $0$ which has a $1$ either in its row or column (or both).}$$ 

\item $\mathsf{exc}$ and $\mathsf{wex}$. Write $\pi$ in the matrix notation. 
We say $i$ is an excedance if $\pi_{ij}=1$ and $i< j$, and a weak excedance if $\pi_{ij}=1$ and $i\le j$. 
Then we define the \emph{excedance} and \emph{weak excedance} statistics by
$$\mathsf{exc}(\pi):= \text{number of excedances},$$ 
$$\mathsf{wex}(\pi):= \text{number of weak excedances}.$$
These definitions clearly inherit those from an ordinary permutation. 

\item $\mathsf{rlmin}$. 
A number $\pi_i=j$ is a right-to-left minimum in a ordinary permutation if in its one-line notation $1,2,\dots, j-1$ are all to the left of $j$. This fact can be restated in terms of the matrix notation and we use this to define a right-to-left minimum of a partial permutation: $\pi_{i,j}=1$ contributes $1$ to right-to-left minimum if $j=1$, or for every column $1\le k< j$ there is a $1$ to the upper left of $\pi_{i,j}$. 
The \emph{right-to-left minimum} statistics is defined by 
$$\mathsf{rlmin}(\pi):=\text{number of right-to-left minimums}.$$

\item  $\mathsf{des}$.  Recall that for a word $\mathsf{c}=c_1\dots c_n$ its descent statistic is 
$$\mathsf{des}(\mathsf{c}):=|\{i: c_i>c_{i+1}\}|.$$ 
For $\pi\in \mathfrak{P}_{n,k}$ let $j_1< j_2 < \dots < j_k$ be the numbers in $[n]$ not appearing in the one-line notation of $\pi$
and $\pi$ will be of the form $\pi=\pi_{(1)}\mathsf{X}\pi_{(2)}\mathsf{X}\dots \mathsf{X}\pi_{(k+1)}$, where each $\pi_{(i)}$ is a (possibly empty) word.  
We define the \emph{descent} statistic of $\pi$ to be
$$\mathsf{des}(\pi) := \mathsf{des}(\pi_{(1)}j_1)+  \mathsf{des}(\pi_{(2)}j_2)+\dots + \mathsf{des}(\pi_{(k)}j_k)+\mathsf{des}(\pi_{(k+1)}),$$
where each $\mathsf{des}(\pi_{(i)}j_i)$ is the descent statistic of the word $\pi_{(i)}j_i$.
\end{itemize}

For our running example $\pi=61\mathsf{X}4923\mathsf{X}5$ we have
$\mathsf{fix(\pi)}=1$, $\mathsf{cyc(\pi)}=3$, $\mathsf{inv(\pi)}=18$, $\mathsf{exc(\pi)}=2$, $\mathsf{wex(\pi)}=3$, $\mathsf{rlmin(\pi)}=4$ and
 $$\mathsf{des(\pi)}=\mathsf{des(617)}+\mathsf{des(49238)}+\mathsf{des(5)}=1+1+0=2.$$

Note that we do not define the major index $\mathsf{maj}$, which will be discussed  in the last section.
We use the notation $\mathsf{stat}_1\sim \mathsf{stat}_2$ to mean two statistics are equidistributed 
and use $(\mathsf{stat}_1,\mathsf{stat}_2)$ to denote their joint distribution.
For example, it is well known that over $\mathfrak{S}_n$  we have $\mathsf{inv}\sim \mathsf{maj}$,
$\mathsf{cyc}\sim \mathsf{rlmin}$ and $\mathsf{des}\sim \mathsf{exc}\sim (\mathsf{wex}-1)$. 
Also we have the symmetric joint distribution $(\mathsf{cyc}, \mathsf{rlmin})\sim (\mathsf{rlmin}, \mathsf{cyc})$ 
over $\mathfrak{S}_n$. Readers are referred to~\cite{Bona_12} for more information. A statistic is \emph{Mahonian} if it is equidistributed with $\mathsf{inv}$, is \emph{Stirling} if equidistributed with $\mathsf{rlmin}$, and is \emph{Eulerian} if equidistributed with $\mathsf{des}$.


\section{Inversion, weak excedance and right-to-left minimum}

Our first result reveals that in partial permutations $\mathsf{rlmin}, \mathsf{inv}$ and $\mathsf{wex}$ can be jointly encoded in the seed sequences.  Define $$[n]_{w,q}:=w+q+\dots +q^{n-1} =[n]-1+w$$
by replacing the constant term of $[n]$ by $w$. Note that $[1]_{w,q}=w$ and $[0]_{w,q}=0$ by definition. 

\begin{thm}\label{thm_par_triple}
Let the seed sequences be $s_0=wp$ and 
$$(s_\ell,t_\ell)=\left(([\ell]_{w,q}+p[\ell+1]_q)q^\ell,p[\ell]_{w,q}[\ell]_qq^{2\ell-1}\right)$$ for $\ell\geq 1$. Then 
$a_{n,k}$ of the Catalan matrix is the generating function of $\mathfrak{P}_{n,k}$ with respect to $\mathsf{rlmin}, \mathsf{inv}$ and $\mathsf{wex}$. Namely,
$$a_{n,k}=\sum_{\pi\in \mathfrak{P}_{n,k}}w^{\rlmin(\pi)}p^{\wex(\pi)}q^{\inv(\pi)}.$$
\end{thm}

\begin{proof}
Let
$$b_{n,k}:=\sum_{\pi\in\mathfrak{P}_{n,k}}w^{\rlmin(\pi)}p^{\wex(\pi)}q^{\inv(\pi)}$$
and we will prove $b_{n,k}$ satisfies the same recurrence as $a_{n,k}$ does. Note that $b_{0,0}=1=a_{0,0}$.

The idea of the proof is to partition $\mathfrak{P}_{n,k}=\cup_{r=1}^4 \mathfrak{P}^{(r)}_{n,k}$ into a disjoint union of four sets.
Given $\pi \in \mathfrak{P}_{n,k}$ we say $\pi$ belongs to
\begin{enumerate}
\item $\mathfrak{P}^{(1)}_{n,k}$: if entries in the last row and column are all zeros;
\item $\mathfrak{P}^{(2)}_{n,k}$: if $\pi_{nn}=1$;
\item $\mathfrak{P}^{(3)}_{n,k}$: if $\pi_{nn}=0$, and either entries of the last row or the last column are all zeros;
\item $\mathfrak{P}^{(4)}_{n,k}$: if $\pi_{nn}=0$, and neither the last row nor the last column are all zeros.
\end{enumerate}

We shall compute 
$$b^{(i)}_{n,k}:=\sum_{\pi\in\mathfrak{P}^{(i)}_{n,k}}w^{\rlmin(\pi)}p^{\wex(\pi)}q^{\inv(\pi)}$$ for $i=1,2,3,4$,
and prove that
\begin{eqnarray*}
b^{(1)}_{n,k}&=&b_{n-1,k-1}, \quad \mbox{for } k\ge 0,\\
b^{(2)}_{n,k}&=& \begin{cases} wpb_{n-1,0}, \quad \mbox{for } k=0 \\ pq^{2k}b_{n-1,k}, \quad \mbox{for } k\ge 1  \end{cases}, \\
b^{(3)}_{n,k}&=& ([k]_{w,q}+p[k]_q)q^k b_{n-1,k}, \quad \mbox{for } k\ge 0,\\
b^{(4)}_{n,k}&=& \left(p[k+1]_{w,q}[k+1]_qq^{2k+1}\right)b_{n-1,k+1}, \quad \mbox{for } k\ge 0.
\end{eqnarray*}

We discuss each set one by one:  

\begin{enumerate}
\item If $\pi \in \mathfrak{P}^{(1)}_{n,k}$, delete its last column and row we get a $\pi'\in  \mathfrak{P}_{n-1,k-1}$. On the other hand, given 
$\pi'\in  \mathfrak{P}_{n-1,k-1}$ there exists exactly one $\pi \in \mathfrak{P}^{(1)}_{n,k}$ by attaching a zero column and a zero row.
 It is easy to see that 
$\rlmin(\pi)=\rlmin(\pi')$, $\wex(\pi)=\wex(\pi')$ and $\inv(\pi)=\inv(\pi')$ and hence 
\begin{eqnarray*} 
b^{(1)}_{n,k} &=& \sum_{\pi\in \mathfrak{P}^{(1)}_{n,k}}w^{\rlmin(\pi)}p^{\wex(\pi)}q^{\inv(\pi)}\\
&=& \sum_{\pi'\in \mathfrak{P}_{n-1,k-1}}w^{\rlmin(\pi')}p^{\wex(\pi')}q^{\inv(\pi')}\\
&=& b_{n-1,k-1}.
\end{eqnarray*}

\item 
If $\pi \in \mathfrak{P}^{(2)}_{n,k}$, delete its last column and row we get a $\pi'\in  \mathfrak{P}_{n-1,k}$.
On the other hand, given $\pi'\in \mathfrak{P}_{n-1,k}$ there is exactly one corresponding $\pi\in \mathfrak{P}^{(2)}_{n,k}$.
It is clear that $\wex(\pi)=\wex(\pi')+1$ and by definition $\inv(\pi)=\inv(\pi')+2k$.
Now if $k=0$, then $n$ is a right-to-left minimum and we have $\rlmin(\pi)=\rlmin(\pi')+1$. If $k\ge 1$, 
then $\rlmin(\pi)=\rlmin(\pi')$. Hence 
$$\sum_{\pi\in \mathfrak{P}^{(2)}_{n,k}}w^{\rlmin(\pi)}p^{\wex(\pi)}q^{\inv(\pi)}= 
\begin{cases}
wpb_{n-1,0} & \text{if }k=0,  \\
pq^{2k}b_{n-1,k} & \text{if }k\ge 1.
\end{cases}
$$
 
\item If $\pi \in \mathfrak{P}^{(3)}_{n,k}$, delete its last column and row we get a $\pi'\in  \mathfrak{P}_{n-1,k}$.
On the other hand, given $\pi'\in \mathfrak{P}_{n-1,k}$ there are
$2k$ choices to put a $1$ into the last column or row and result in a $\pi \in \mathfrak{P}^{(3)}_{n,k}$.
Let these possible positions be
 $(n,j_1),(n,j_2),\ldots,(n,j_k)$, where $j_1<j_2<\cdots<j_k<n$  and  $(i_1,n),(i_2,n),\ldots,(i_k,n)$, where $i_1<i_2<\cdots<i_k<n$.

First let us see if $\pi_{n, j_\ell}=1$ for some $1\leq \ell\leq k$. By the definition of the inversion, from $\pi'$ to $\pi$  there will be $k$ extra inversions in the $j_\ell$-th column and $\ell-1$ extra inversions in the last row. So 
$$\inv(\pi)=\inv(\pi')+k+\ell-1.$$
As $j_\ell$ is a right to left minimum if and only if $\ell=1$, we have $\rlmin(\pi)=\rlmin(\pi')+1$ if $\ell=1$, and  $\rlmin(\pi)=\rlmin(\pi')$ if $\ell>1$. 
There is no extra weak excedance and $\wex(\pi)=\wex(\pi')$. Adding these up we know that in this case from $\pi'$ to $\pi$ we should multiply a factor
$$wq^k+q^{k+1}+q^{k+2}+\dots q^{2k-1}=[k]_{w,q}q^k$$ in the generating function.

On the other hand, if $\pi_{i_\ell,n}=1$ for some $1\leq \ell\leq k$, then 
$\inv(\pi)=\inv(\pi')+k+\ell-1$, $\rlmin(\pi)=\rlmin(\pi')$, and $\wex(\pi)=\wex(\pi')+1$ by similar discussion. 
So in this case from $\pi'$ to $\pi$ we should multiply a factor $$pq^k+pq^{k+1}+\dots +pq^{2k-1}=pq^k[k]_q.$$ Hence
\begin{eqnarray*} 
b^{(3)}_{n,k} &=& \sum_{\pi\in \mathfrak{P}^{(3)}_{n,k}}w^{\rlmin(\pi)}p^{\wex(\pi)}q^{\inv(\pi)}\\
                      &=& \left( ([k]_{w,q}+p[k]_q)q^k \right)b_{n-1,k}.
\end{eqnarray*}
Also $b^{(3)}_{n,0}=0$ by definition.
\item If $\pi \in \mathfrak{P}^{(4)}_{n,k}$, 
delete its last column and row we get a partial permutation in $\mathfrak{P}_{n-1,k+1}$.
On the other hand, given $\pi'\in \mathfrak{P}_{n-1,k+1}$ there are
$2k+2$ positions to put $1$ into the last column or row and result in a $\pi \in \mathfrak{P}^{(4)}_{n,k}$,
say $(n,j_1),(n,j_2),\ldots,(n,j_{k+1})$ with $j_1<j_2<\cdots<j_{k+1}<n$, and $(i_1,n),(i_2,n),\ldots,(i_{k+1},n)$ with $i_1<i_2<\cdots<i_{k+1}<n$. 

 Let $\pi_{n,j_\ell}=1$ for some $1\le \ell \le k+1$ and $\pi_{i_{\ell'}, n}=1$ for some $1\leq \ell'\leq k+1$. 
Similar to the discussion above, we have $\inv(\pi)=\inv(\pi')+2k+1+(\ell-1)+(\ell'-1)$;
$\rlmin(\pi)=\rlmin(\pi')+1$ if $\ell=1$ and $\rlmin(\pi)=\rlmin(\pi')$ if $\ell>1$. Also $\wex(\pi)=\wex(\pi')+1$.
Hence 
\begin{eqnarray*}
b^{(4)}_{n,k} &=& \sum_{\pi\in \mathfrak{P}^{(4)}_{n,k}}w^{\rlmin(\pi)}p^{\wex(\pi)}q^{\inv(\pi)}\\
 &=& p[k+1]_{w,q}[k+1]_qq^{2k+1} b_{n-1,k+1}.
\end{eqnarray*}
\end{enumerate}
Therefore, from (1)-(4) we have
\begin{eqnarray*}
b_{n,0} &=& b^{(1)}_{n,0}+b^{(2)}_{n,0}+b^{(3)}_{n,0}+b^{(4)}_{n,0}\\
 &=& 0+wpb_{n-1,0}+0+p[1]_{w,q}[1]_qqb_{n-1,1}\\
&=& s_0b_{n-1,0}+t_{1}b_{n,1},
\end{eqnarray*}
and if $k\ge 1$ we have
\begin{eqnarray*}
b_{n,k} &=& b^{(1)}_{n,k}+b^{(2)}_{n,k}+b^{(3)}_{n,k}+b^{(4)}_{n,k}\\
 &=& b_{n-1,k-1}+ pq^{2k}b_{n-1,k}+ \left([k]_{w,q}+p[k]_q\right) q^kb_{n-1,k} + \left(p[k+1]_{w,q}[k+1]_qq^{2k+1}\right)b_{n-1,k+1}\\
&=& b_{n-1,k-1}+ \left([k]_{w,q}+p[k+1]_q\right)q^k b_{n-1,k}+\left(p[k+1]_{w,q}[k+1]_qq^{2k+1}\right)b_{n-1,k+1}\\
&=& b_{n-1,k-1}+s_kb_{n-1,k}+t_{k+1}b_{n,k+1},
\end{eqnarray*}
which is exactly the same recurrence relation satisfied by $a_{n,k}$ and the proof is completed. 
\end{proof}


\begin{table}[t]\label{rlminwexinv}
\begin{center}
\begin{tabular}{ c|cccc } 
$n \backslash k$ & 0& 1& 2 & 3\\ \hline
0 & 1& ~& ~ & ~\\ \hline
1 & $pw$ & 1 & ~ & ~\\ \hline
2 & $p^2w^2+pqw$ 
&   
$pq^2+pq+pw+qw$
& 1 & ~\\ \hline
3 & 
$\begin{array}{c} 
p^3w^3+p^2q^3w+p^2q^2w\\
\quad +2p^2qw^2+pq^2w^2
\end{array}$
 & 
$\begin{array}{c} 
 p^2q^4+pq^5+pq^4w+2p^2q^3+p^2q^2w\\
+pq^4+3pq^3w+p^2q^2+p^2qw+p^2w^2\\
+2pq^2w+pqw^2+q^2w^2+pqw
\end{array}$
 &
$\begin{array}{c} 
pq^4+pq^3+2pq^2\\
+q^3+q^2w+pq\\
+pw+qw
\end{array}$
& 1\\
\hline
\end{tabular}
\end{center}
\caption{Partial permutations with $\mathsf{rlmin}, \mathsf{wex}, \mathsf{inv}$.}
\end{table}

Table 1 presents some initial values.
By putting $q=1$, $w=1$ or $p=1$ we may obtain several specializations.
Among them the bivariate generating function with respect to $\mathsf{rlmin}$ and $\mathsf{inv}$
has a nice closed form.

\begin{thm}\label{thm_par_inv_rlmin}
Let $n,k$ be nonnegative integers. Then
$$\sum_{\pi\in \mathfrak{P}_{n,k}}w^{\rlmin(\pi)}q^{\inv(\pi)}=\left[{n\atop k}\right][n]_{w,q}[n-1]_{w,q}\cdots [k+1]_{w,q}.$$
\end{thm}
\begin{proof}
For $\pi\in \mathfrak{P}_{n,k}$ in its matrix notation, let 
$$S=\{i :\text{the $i$-th row of $\pi$ is nonzero}\}=\{i_1,i_2,\ldots,i_{n-k}\}$$
 with $i_1< i_2<\cdots < i_{n-k}$ being the row indices of its $1$'s and let $T=[j_1,j_2,\ldots, j_{n-k}]$ be the $(n-k)$-permutation of $\{1,2,\ldots,n\}$ in the one-line notation
such that $(i_1, j_1), (i_2, j_2), \ldots,(i_{n-k}, j_{n-k})$ are the positions of $1$'s.
Conversely, given a pair $(S, T)$ for an $(n-k)$-subset $S$ and an $(n-k)$-permutation $T$ of $\{1,2,\ldots,n\}$, 
we can determine $\pi$ uniquely.

For this $\pi$, let us first count the inversions on those rows whose indices are not in $S$. Clearly these inversions are possible only
at the $j_1, j_2,\dots , j_{n-k}$-th columns and there are $i_\ell-\ell$ such inversions at the $j_\ell$-th column for $1\le \ell \le n-k$. Therefore, 
the number of such inversions equals the inversion $\mathsf{inv}(\mathbf{c}_S)$ of the binary word $\mathbf{c}_S=c_1c_2\dots c_n$, 
where $c_\ell=0$ if $\ell \in S$ and $c_\ell =1$ otherwise. Note that $\mathbf{c}_S$ only depends on the choice of $S$.


Note that for any inversion of $\pi$ not enumerated yet, there is a $1$ to its right. 
Now we fix an $(n-k)$-subset $S$. We put a $1$ from top to bottom for each row indexed by $S$ and compute the inversions to its left and the right-to-left minimums it may produce. It is clear that for $1\le \ell\le n-k$, there are $n+1-\ell$ choices of positions to put a $1$ in the $i_\ell$-th row.
Among these positions, only if we choose the leftmost feasible position will it contributes $1$ to $\rlmin(\pi)$. Also, if we choose the $r$-th feasible position from left to right, it contributes $r-1$ to $\inv(\pi)$. Hence we arrive at the generating function 
$$\prod_{\ell=1}^{n-k}(w+q+q^2+\cdots+q^{n-\ell}).$$
Note that it only depends on the size of $S$. Therefore, we have
\begin{eqnarray*}
\sum_{\pi\in \mathfrak{P}_{n,k}}w^{\rlmin(\pi)}q^{\inv(\pi)}&=&\sum_{S\subseteq [n], |S|=n-k} \left(q^{\inv(\mathbf{c}_S)}\prod_{\ell=1}^{n-k}(w+q+q^2+\cdots+q^{n-\ell})\right)\\
&=&\left(\sum_{S\subseteq [n], |S|=n-k} q^{\inv(\mathbf{c}_S)} \right) \left( \prod_{\ell=1}^{n-k}(w+q+q^2+\cdots+q^{n-\ell})\right)\\
&=&\left[{n\atop k}\right][n]_{w,q}[n-1]_{w,q}\cdots[k+1]_{w,q}
\end{eqnarray*}
and the theorem is proved.
\end{proof}


By letting $w=1$ and $q=1$ respectively we obtain the following corollaries, generalizing the classic results in $\mathfrak{S}_n$.
\begin{cor}\label{cor_par_inv}
Let $n,k$ be nonnegative integers, Then
$$\sum_{\pi\in \mathfrak{P}_{n,k}}q^{\inv(\pi)}=\left[{n\atop k}\right]^2[n-k]!$$
\end{cor}
\begin{cor}\label{cor_par_rlmin}
Let $n,k$ be nonnegative integers, Then
$$\sum_{\pi\in \mathfrak{P}_{n,k}}w^{\rlmin(\pi)}=\left({n\atop k}\right)(w+k)(w+k+1)\cdots(w+n-1).$$
\end{cor}

An implication of the Corollary~\ref{cor_par_inv} (respectively Corollary~\ref{cor_par_rlmin}) is that other Mahonian (respectively Stirling) statistics  over $\mathfrak{P}_{n,k}$ should share the same generating function with respect to $\inv$ (respectively $\rlmin$). For Eulerian statistics, a possible checkpoint is the equidistribution with $\wex$. One can observe several interesting patterns in Table 2. 

\begin{table}[h!]\label{wextable}
\begin{center}
\begin{tabular}{ c|ccccc } 
 $n\backslash k$ & 0&  1 & 2 & 3 &4 \\ \hline
 0 & 1 & ~ & ~ & ~& ~\\
 1 & $p$ & $1$ & ~ & ~& ~\\
 2 & $p^2+p$ & $3p+1$ & $1$ & ~& ~\\
 3 & $p^3+4p^2+p$ & $7p^2+10p+1$ & $6p+3$ & $1$& ~\\
 4 & $p^4+11p^3+11p^2+p$ & $15p^3+55p^2+25p+1$ & $25p^2+40p+7$ & $10p+6$& $1$\\
\end{tabular}
\end{center}
\caption{Generating functions of partial permutations by $\mathsf{wex}$.}
\end{table}

\section{Descent, excedance and weak excedance}
In this section we prove that the fundamental result $\exc \sim \des $ over $\mathfrak{S}_n$ remains true for partial permutations, 
and $\exc \sim (\wex-1)$ is also true with a new point of view.

\subsection{Excedance and weak excedance}
We first prove that $\exc$ and $\wex$ are `equidistributed' in the following sense.

\begin{thm}\label{prop_exc_wex}
Let $n,k$ be nonnegative integers. Then the number of partial permutations $\pi \in \mathfrak{P}_{n,k}$ with $i$ excedances, $1\le i\le n-k$, equals those partial permutations with $n-k-i$ weak excedances. Namely, if we let 
$$f_n^{(k)}(q)=\sum_{\pi\in \mathfrak{P}_{n,k}}q^{\exc}\quad \text{and} \quad g_n^{(k)}(q)=\sum_{\pi\in \mathfrak{P}_{n,k}}q^{\wex},$$ Then 
for $1\leq i\leq n-k$ we have
$$[q^i]f_n^{(k)}(q)=[q^{n-k-i}]g_n^{(k)}(q).$$ 
\end{thm}
\begin{proof}
Let  $\phi$ be the bijection on $\mathfrak{P}_{n,k}$ such that $\phi(\pi)$ is 
obtained by rotating the matrix form of $\pi$ by $180^\circ$. It is easy that if $\exc(\pi)=i$ then $\wex(\phi(\pi))=n-k-i$ and the result follows.
\end{proof}

From this theorem the number of ordinary permutations with $r+1$ weak excedances equals those with $n-r-1$ excedances, hence equals those with $r$ excedances since $\sum_{\pi\in \mathfrak{S}_n}q^{\exc}=\sum_{\pi\in \mathfrak{S}_n}q^{\des}$ and the latter is palindromic. This offers an interesting way to understand the classic result $(\wex-1)\sim \exc$.

\subsection{Descent and excedance}
A key for a combinatorial proof of $\des \sim \exc$ over $\mathfrak{S}_n$ is the fundamental bijection (\cite{Stanley_EC1}, p.23).\\

\noindent \textit{Fundamental bijection $\phi: \mathfrak{S}_n \to \mathfrak{S}_n$}. Given $\pi =\pi_1\pi_2\dots \pi_n\in \mathfrak{S}_n$ we perform the following steps.
\begin{enumerate}
\item[(i)]  Put a vertical bar to the right of each right-to-left minimum of $\pi$. Hence $\pi$ is cut into segments.
\item[(ii)]  Regard each segment as a cycle, resulting in the cycle notation of a new permutation $\phi(\pi)$.
\end{enumerate}
 For example, the bijection maps $\pi=316425$ to $\phi(\pi)=(31)(642)(5)=361254$ via
$$316425\to 31|642|5| \to (31)(642)(5).$$
To perform the inverse, first we write a permutation in its cycle notation and arrange cycles in a way that 
in each cycle the smallest element is put at the last and these smallest elements are in increasing order among cycles. 
The desired permutation is obtained by dropping off all parentheses.
Note that $\phi$ can be extended to a subsequence of a permutation since the above steps still make sense once we can locate the right-to-left minimums.

Now we generalize the fundamental bijection to partial permutations.\\

\noindent \textit{Generalized fundamental bijection $\phi: \mathfrak{P}_{n,k} \to \mathfrak{P}_{n,k}$.}
Given $\pi\in \mathfrak{P}_{n,k}$ in the one line notation and let 
$i_1<i_2<\cdots<i_k$ be the numbers not appearing in $\pi$.
We perform the following steps.
\begin{enumerate}
\item[(i)] Write $\pi$ as 
$$\pi=\pi_{(1)}\mathsf{X}\pi_{(2)}\mathsf{X}\cdots \mathsf{X}\pi_{(k+1)},$$ where each $\pi_{(i)}$ is a (possibly empty) word.
\item[(ii)] Perform Fundamental bijection (only) on $\pi_{(k+1)}$ and result in $\tilde{\pi}_{(k+1)}$.
\item[(iii)] Define $$\phi(\pi):=(\pi_{(1)}i_1\mathsf{X})(\pi_{(2)}i_2\mathsf{X})\cdots(\pi_{(k)} i_k\mathsf{X})\tilde{\pi}_{(k+1)},$$
where $\tilde{\pi}_{(k+1)}$ is in the cycle notation.
\end{enumerate}

For instances, we have $\phi(1\mathsf{X}245) = (13\mathsf{X})(2)(4)(5)$, $\phi(1\mathsf{X}3\mathsf{X}756)= (12\mathsf{X})(34\mathsf{X})(75)(6),$ and 
$\phi(32\mathsf{X}\mathsf{X}186\mathsf{X})=(324\mathsf{X})(5\mathsf{X})(1867\mathsf{X})$.

The inverse mapping $\phi^{-1}$ can be done as follows. Given $\pi \in \mathfrak{P}_{n,k}$, 
we write its cycle notation in its \textit{standard form} 
$$\pi= c'_1c'_2\cdots c'_kc_1c_2\cdots c_{\cyc(\pi)},$$ 
where $c'_1,\ldots, c'_k$ are partial cycles in which the last numbers among them are increasing,
and $c_1,\ldots, c_{\cyc(\pi)}$ are full cycles with the convention that 
in each $c_i$ we put its smallest number to be the last one and we arrange $c_i$ in increasing order with respect to these last numbers.
Then $\phi^{-1}(\pi)$ is obtained by 
deleting the last number for each $c'_i$ and erase all parentheses. 

We can now prove the equidistribution of descents and excedances.
\begin{thm}\label{thm_exc_des}
Let $n,k$ be nonnegative integers. Then $$\des\sim \exc$$ over $\mathfrak{P}_{n,k}$.
\end{thm}
\begin{proof}
We will give a bijection $\varphi$ on $\mathfrak{P}_{n,k}$ such that $\exc(\pi)=\des(\varphi(\pi))$ for all $\pi\in \mathfrak{P}_{n,k}$.  
Given $\pi\in \mathfrak{P}_{n,k}$, we perform the following steps.
\begin{enumerate}
\item[(i)] Write $\pi$ in the cycle notation (the order of the cycles does not matter).
\item[(ii)] Reverse entries of each cycle but keeping $\mathsf{X}$ unmoved. 
\item[(iii)] Arrange these cycles into the standard form.
\item[(iv)] Perform $\phi^{-1}$ and result in $\varphi(\pi)$. 
\end{enumerate}
For example, take $\pi=4\mathsf{X}35\mathsf{X}\mathsf{X}2=(3)(72\mathsf{X})(6\mathsf{X})(145\mathsf{X})$. Step 2 leads to $(3)(27\mathsf{X})(6\mathsf{X})(541\mathsf{X})$,
and Step 3 leads to $(541\mathsf{X})(6\mathsf{X})(27\mathsf{X})(3)$. Therefore
$$\varphi(4\mathsf{X}35\mathsf{X}\mathsf{X}2)=\phi^{-1}( (541\mathsf{X})(6\mathsf{X})(27\mathsf{X})(3))=54\mathsf{X}\mathsf{X}2\mathsf{X}3.$$
Readers may try another example $\varphi(1376\mathsf{X}42)=\mathsf{X}173264$. It is direct to verify that $\exc(\pi)=\des(\varphi(\pi))$ and the proof is done.
\end{proof}

\section{Extended Catalan matrix and set-valued statistics}
The notion of set-valued statistic over permutations is a natural generalization of the ordinary integer-valued statistic and have attracted much attention recently, see~\cite{Eu_Lo_Wong_15, Foata_Han_09, Lin_Kim_18, Mao_Zeng_21, Poznanovic_14} for examples.
In this section, we shall introduce the notion of extended Catalan matrix, generated from two 2-dimensional seed sequences, to help us deal with set-valued statistics over partial permutations.

Given two 2-dimensional sequences $\mathbf{s}=(s^{(i)}_\ell)_{0\leq \ell, 1\leq i}$ and $\mathbf{t}=(t^{(i)}_\ell)_{1\leq \ell, 1\leq i}$, we define an infinite lower triangular matrix  $A^{\mathbf{s},\mathbf{t}}=(a_{n,k})_{n,k\geq 0}$ in a similar way from the recurrence
$$
\begin{cases}
a_{0,0} = 1, a_{0,k}=0 \quad (k> 0), \\
a_{n,k} = a_{n-1,k-1}+s^{(n)}_ka_{n-1,k}+t^{(n)}_{k+1}a_{n-1,k+1}\quad(n\geq 1),
\end{cases}
$$
where $a_{n, -1}$ is set to be zero. The matrix $A^{\mathbf{s},\mathbf{t}}$ is called the {\it extended Catalan matrix} associated to the {\it extended seed sequences} (or {\it extended seeds} for short) $\mathbf{s}$ and $\mathbf{t}$. We will encode some set-valued statistics over partial permutations in the extended seeds.

Given $\pi\in \mathfrak{P}_{n,k}$ in the matrix notation. We define the following set-valued statistics: 

\begin{itemize}
\item $\Fix(\pi):=\{i:\pi_{ii}=1\}$, the set of fixed points.
\item $\Sfix(\pi)$. We call $i$ a \emph{strong fixed point of $\pi$} if $\pi_{ii}=1$ and the principal submatrix of $\pi$ formed by the first $i$ rows and columns is a permutation matrix. Let
$$\mathsf{Sfix}(\pi):= \{i:\text{ $i$ is a strong fixed point}\}$$
be the set of strong fixed points and $\sfix(\pi):=|\Sfix(\pi)|$.
\item $\Cyc(\pi):=\{i: \text{$i$ is the largest number in a full cycle}\}$.
\item $\Cyc_{\geq 2}(\pi):=\Cyc(\pi)-\Fix(\pi)$.
\item $\Scyc(\pi)$. Let $C$ be a full cycle of $\pi$ and $i$ is the largest number in $C$. If the principal submatrix of $\pi$ formed by the first $i$ rows and columns is a permutation matrix, then $C$ is called  a \emph{strong cycle}. Define
$$\Scyc(\pi):=\{i: \text{$i$ is the largest number in a strong cycle}\} $$ 
and $\scyc(\pi):=|\Scyc(\pi)|$
\item $\Scyc_{\geq 2}(\pi):=\Scyc(\pi)-\Sfix(\pi)$.
\item $\Excl(\pi):=\{j:\text{$\pi_{ij}=1$ and $i<j$}\}$.
\item $\Rlmip(\pi):=\{i: \text{$\pi_{ij}=1$ and $j$ is a right-to-left minimum of $\pi$}\}$.
\end{itemize}

Note that for $\pi\in \mathfrak{P}_{n,0}$, $\Scyc(\pi)-\{n\}=\Sfix(\pi)\cup\Scyc_{\geq 2}(\pi)-\{n\}$ is the connectivity set of permutation $\pi$ introduced by Stanley \cite{Stanley_05}.
Therefore we define the \emph{connectivity set} of $\pi \in \mathfrak{P}_{n,k}$ to be $\mathsf{Scyc}(\pi)-\{n\}$.
Also we define $$\mathbf{x}^S:=\prod_{i\in S}x_i $$
for a set of integers $S$.
\medskip

We give an example to demonstrate the interests of set-valued statistics.
By letting $p=q=1$ in Theorem \ref{thm_par_triple} we know that by taking seed sequences $s_0=w$ and
$(s_\ell,t_\ell)=(2\ell+w,\ell(\ell-1+w))$ for $\ell\geq 1$, we have 
$$a_{n,k}=\sum_{\pi\in\mathfrak{P}_{n,k}}w^{\rlmin(\pi)}.$$

The corresponding set-valued analogue is the following.
\begin{thm}\label{thm_par_Rlmip}
Let the seed sequences be $s^{(i)}_0=w_i$ and
$$(s^{(i)}_\ell,t^{(i)}_\ell)=(2\ell+w_i,\ell(\ell-1+w_i))$$
for $\ell\geq 1$. Then $a_{n,k}$ of the extended Catalan matrix is the generating function of $\mathfrak{P}_{n,k}$ with respect to $\mathsf{\Rlmip}$. Namely,
$$a_{n,k}=\sum_{\pi\in\mathfrak{P}_{n,k}}\mathbf{w}^{\Rlmip(\pi)}.$$
\end{thm}

In fact, the main result of this section is to prove the following generalization.

\begin{thm}\label{thm_par_cyc_rlmin}
Let the seed sequences be $s^{(i)}_0=a_iw_iu_i$, $t^{(i)}_1=b_iw_ip_iv_i$ and 
\begin{eqnarray*}
s^{(i)}_\ell=&\ell-1+a_i+w_i+\ell p_i &\text{ for $\ell\geq 1$},\\ 
t^{(i)}_\ell=&(\ell-1+b_i)(\ell-1+w_i)p_i&\text{ for $\ell\geq 2$}. 
\end{eqnarray*}
Then 
$a_{n,k}$ of the extended Catalan matrix is the generating function of $\mathfrak{P}_{n,k}$ with respect to $\mathsf{Fix}, \mathsf{Cyc}_{\ge 2}$, $\mathsf{\Rlmip}$, $\Excl$, $\Sfix$ and $\Scyc$. Namely,
$$a_{n,k}=\sum_{\pi\in \mathfrak{P}_{n,k}}\mathbf{a}^{\Fix(\pi)}\mathbf{b}^{\Cyc_{\ge 2}(\pi)}\mathbf{w}^{\Rlmip(\pi)}\mathbf{p}^{\Excl(\pi)}\mathbf{u}^{\Sfix(\pi)}\mathbf{v}^{\Scyc(\pi)}.$$
\end{thm}

\begin{proof}
The idea of the proof is similar to that of Theorem~\ref{thm_par_triple}. Let 
$$b_{n,k}=\sum_{\pi\in \mathfrak{P}_{n,k}}\mathbf{a}^{\Fix(\pi)}\mathbf{b}^{\Cyc_{\ge 2}(\pi)}\mathbf{w}^{\Rlmip(\pi)}\mathbf{p}^{\Excl(\pi)}\mathbf{u}^{\Sfix(\pi)}\mathbf{v}^{\Scyc(\pi)}.$$
We will show that it satisfies the same recurrence as $a_{n,k}$ does.
By definition $b_{0,0}=1=a_{0,0}$.
We partition  $\mathfrak{P}_{n,k}$ into four disjoint sets 
$\mathfrak{P}_{n,k}=\cup_{r=1}^4 \mathfrak{P}^{(r)}_{n,k}$ 
and define 
$$b^{(r)}_{n,k} = \sum_{\pi\in \mathfrak{P}^{(r)}_{n,k}}\mathbf{a}^{\Fix(\pi)}\mathbf{b}^{\Cyc_{\ge 2}(\pi)}\mathbf{w}^{\Rlmip(\pi)}\mathbf{p}^{\Excl(\pi)}\mathbf{u}^{\Sfix(\pi)}\mathbf{v}^{\Scyc(\pi)},$$ 
$1\le r\le 4$,
as in the proof of Theorem \ref{thm_par_triple}.
It is easy to see that
$$b^{(1)}_{n,k} =b_{n-1,k-1}$$
and 
$$
b^{(2)}_{n,k} =
\begin{cases}
a_n\cdot w_n\cdot u_n\cdot b_{n-1,0},  &\text{ if } k=0, \\
a_n\cdot b_{n-1,k},  &\text{ if } k\geq 1.
\end{cases}
$$

For $\mathfrak{P}^{(3)}_{n,k}$ we proceed as follows. 
 If $k=0$, then $\mathfrak{P}^{(3)}_{n,k}=\emptyset$ and $b^{(3)}_{n,k}=0$.
Suppose $k\ge 1$. Given $\pi \in \mathfrak{P}^{(3)}_{n,k}$, by deleting its last column and row we obtain a $\pi' \in \mathfrak{P}_{n-1,k}$.
On the other hand, given $\pi' \in \mathfrak{P}_{n-1,k}$, the $2k$ possible positions to put an $1$ in the appended new last column or row to result in a $\pi \in \mathfrak{P}^{(3)}_{n,k}$ can be described in terms of the cycle decomposition of $\pi'$. Namely, let the $k$ partial cycles of $\pi'$ be
$$(h_1\dots h_1' \mathsf{X}), (h_2\dots h_2'\mathsf{X}),\dots, (h_k\dots h_k'\mathsf{X}),$$
with $h_1<h_2<\dots <h_k$. Then it is clear that the positions of the appended $1$ can be at
$$(n, h_1), (n, h_2), \dots , (n, h_k), (h_1',n), (h_2',n),\dots ,(h_k',n).$$
It is easy to see that $n\in\Excl(\pi)$ if and only if the appended 1 is at $(h'_i,n)$ for $1\leq i\leq k$. 
Note that the number of the full cycles keeps the same no matter where this $1$ is, and only 
when it is at $(n, h_1)$ does it result in $n\in\Rlmip(\pi)$. Hence we have 
$$b^{(3)}_{n,k} =
\begin{cases}
0,  &\text{if }k=0,\\
(w_n+k-1+kp_n)b_{n-1,k},  &\text{if }k\ge 1.
\end{cases}
$$

If $\pi \in \mathfrak{P}^{(4)}_{n,k}$, 
delete its last column and row we get a partial permutation $\pi'$ in $\mathfrak{P}_{n-1,k+1}$.
On the other hand, given $\pi'\in \mathfrak{P}_{n-1,k+1}$, there are
$2k+2$ positions to put $1$ into the appended new last column or row to form a $\pi \in \mathfrak{P}^{(4)}_{n,k}$.
Similar to the discussion of the above case, we let the $$(h_1\dots h_1' \mathsf{X}), (h_2\dots h_2'\mathsf{X}),\dots, (h_{k+1}\dots h_{k+1}'\mathsf{X})$$
be the partial cycles of $\pi'$ with $h_1<h_2<\dots <h_{k+1}$. We need to put 1's in two positions $(n,h_i)$ and $(h_j',n)$ for some $1\le i,j\le k+1$. 
For each choice we have $n\in\Excl(\pi)$. 
Note that if $i=j$, then $\Cyc(\pi)=\Cyc(\pi')\cup\{n\}$ and the new full cycle has length at least $2$. Otherwise $i\ne j$ and 
$\Cyc(\pi)=\Cyc(\pi')$. Moreover, only if the new $1$ is at $(n, h_1)$ does it result in $n\in\Rlmip(\pi)$. 
If $k=0$, clearly there is the only choice $i=j=1$ and $n\in \Scyc(\pi)$. 
Hence we have
$$b^{(4)}_{n,k} = 
\begin{cases}
b_n\cdot r_n\cdot p_n\cdot v_n\cdot b_{n-1,1},  &\text{if }k=0,\\
(w_n+k)(b_n+k)p_n\cdot b_{n-1,k+1}, &\text{if }k\geq 1.
\end{cases}
$$

Therefore, by adding four cases for $k=0$ we have
$$b_{n,0}= a_nw_nu_nb_{n-1,0} + b_nw_np_nv_nb_{n-1,1}$$
while for $k\ge 1$ we have
$$b_{n,k}= b_{n-1,k-1}+(a_n+w_n+k-1+kp_n)b_{n-1,k}+(b_n+k)(w_n+k)p_nb_{n-1,k+1}.$$
These share the same recurrences as $a_{n,k}$ does and hence the result follows.
\end{proof}

Since a fixed point is a cycle of length $1$, by letting $a_n=b_n$ for each $n$ we have the following
enumerator for $\Cyc$, $\Rlmip$, $\Excl$, $\Sfix$ and $\Scyc$.
\begin{cor}\label{cor_par_cyc_rlmin}
Let the extended seeds be $s^{(i)}_0=a_iw_iu_i$, $t^{(i)}_1=a_iw_ip_iv_i$ and $$(s^{(i)}_\ell, t^{(i)}_{\ell+1}) =(\ell-1+a_i+w_i+\ell p_i, (a_i+\ell)(w_i+\ell)p_i)$$ for $\ell\ge 1$. Then
$$a_{n,k}=\sum_{\pi \in \mathfrak{P}_{n,k}} \mathbf{a}^{\Cyc(\pi)}\mathbf{w}^{\Rlmip(\pi)}\mathbf{p}^{\Excl(\pi)}\mathbf{u}^{\Sfix(\pi)}\mathbf{v}^{\Scyc(\pi)}$$
in the extended Catalan matrix. 
\end{cor}
Note that the parameters $\mathbf{a},\mathbf{w}$ are symmetric. That is, if we exchange $a_i$ and $w_i$ for each $i$, then the extended seeds keep unchanged. Hence we have the following result, generalizing the well-known result that over $\mathfrak{S}_n$ the cycles and the right-to-left minimums have the symmetric joint distribution.

\begin{thm}\label{thm_symmetric}
For any $n,k\ge 0$, we have
$$(\Cyc, \Rlmip, \Excl, \Sfix, \Scyc)\sim (\Rlmip, \Cyc, \Excl, \Sfix, \Scyc)$$
over $\mathfrak{P}_{n,k}$. 
In particular, $\Cyc$ and $\Rlmip$ are equidistributed and have symmetric joint distribution over $\mathfrak{P}_{n,k}$. That is, 
$$(\Cyc, \Rlmip)\sim (\Rlmip, \Cyc).$$
\end{thm}
Theorem \ref{thm_symmetric} generalizes Foata and Han's \cite{Foata_Han_09} result that $(\Cyc, \Rlmip)\sim (\Rlmip, \Cyc)$ over $\mathfrak{S}_n$. 
\medskip

By letting $a_i=a$, $b_i=b$, $w_i=w$, $p_i=p$, $u_i=u$ and $v_i=v$ for all $i$ in Theorem \ref{thm_par_cyc_rlmin}, we can also encode statistics cycle, right-to-left minimum, excedence, strong fixed point and non-singleton strong cycle in the seeds of a Catalan matrix.
\begin{cor}
Let the seeds be $s_0=awu$, $t_1=bwpv$ and $(s_\ell, t_{\ell+1}) =(\ell-1+a+w+\ell p, (b+\ell)(w+\ell)p)$ for $\ell\ge 1$. Then
$$a_{n,k}=\sum_{\pi \in \mathfrak{P}_{n,k}}  {a}^{\fix(\pi)}b^{\cyc_{\geq 2}(\pi)}{w}^{\rlmin(\pi)}{p}^{\exc(\pi)}{u}^{\sfix(\pi)}{v}^{\scyc(\pi)}$$
in the Catalan matrix. 
\end{cor}

For completeness' sake we give a bijective proof of  $\cyc\sim \rlmin$ over $\mathfrak{P}_{n,k}$.
To this end we define a new statistic
\emph{star right-to-left minimum} $\rlmin^*$.
Given  $\pi\in \mathfrak{P}_{n,k}$ , let $\pi= \pi_{(1)}X\pi_{(2)}X \cdots X\pi_{(k+1)}$ ($\pi_{(i)}$ could be empty) be its one-line notation and define
$$\rlmin^*(\pi):=\text{the number of right-to-left minimums of the word } \pi_{(k+1)},$$
and $\rlmin^*(\pi)=0$ if $\pi_{(k)}=\emptyset$.
\begin{proof}  We will show that $$\rlmin \sim \rlmin^* \sim \cyc.$$
Two easy observations of $\rlmin$ and $\rlmin^*$ on the matrix notation of $\pi$ are useful. Both are immediate from definition.
\begin{enumerate}
\item The entry $\pi_{i,j}=1$ contributes $1$ to $\rlmin(\pi)$ if and only if 
there is no zero column before the $j$-th column and all entries to the left or below position $(i,j)$ are zeros.
\item The entry $\pi_{i,j}=1$ contributes $1$ to $\rlmin^*(\pi)$ if and only if 
there is no zero row after the $i$-th row, and all entries to the left or below position $(i,j)$ are zero.
\end{enumerate}
We first prove $\rlmin\sim\rlmin^*$. This is done by reflecting $\pi$ of the matrix notation with respect to its antidiagonal.
Let $\psi(\pi)$ be the reflection. Then by the observations above,
The entry $(i,j)=1$ contributes $1$ to $\rlmin$ of $\pi$ if and only if $(n+1-j, n+1-i)=1$ contributes $1$ to $\rlmin^*$ of $\psi(\pi)$.
Hence $\rlmin(\pi)=\rlmin^*(\psi(\pi))$ and $\rlmin\sim\rlmin^*$.

The fact $\rlmin^*\sim \cyc$ is easier. It is a direct consequence from the definition of the generalized fundamental bijection $\phi(\pi)$ 
as in $\phi$ we know that $\pi_{(i)}$ does not contribute any cycle for $1\le i\le k$ and all manipulations only involve $\pi_{(k+1)}$. Hence the desired conclusion is proved. 
\end{proof}


\section{Connected partial permutations}
Different seeds may result in matrices whose entries count interesting permutation families. In this and the next sections we present two nontrivial instances, namely the connected permutations and cycle-up-down permutations. We start with the connected partial permutations.
\medskip

Recall that a permutation $\pi \in\mathfrak{S}_n$ is connected if $n$ is the least number $i$ such that 
$\{\pi_1,\pi_2,\dots ,\pi_i\}=\{1,2,\dots , i\}$. Similarly, we say 
a partial permutation $\pi\in \mathfrak{P}_{n,k}$ is {\it connected} if in its matrix notation there is no $1\leq i<n$ such that the principal submatrix formed by the first $i$ rows and columns is a permutation of $\{1,2,\ldots,i\}$. 
Let $\mathfrak{C}_{n,k}$ be the set of connected partial permutations in $\mathfrak{P}_{n,k}$ and 
$$C_{n,k}(\mathbf{Z}):=\sum_{\pi\in \mathfrak{C}_{n,k}}\mathbf{z}_1^{\Stat_1(\pi)}\mathbf{z}_2^{\Stat_2(\pi)}\dots \mathbf{z}_r^{\Stat_r(\pi)},$$ 
$$C_{n,k}(\mathbf{z}):=\sum_{\pi\in \mathfrak{C}_{n,k}}z_1^{\stat_1(\pi)}z_2^{\stat_2(\pi)}\dots z_r^{\stat_r(\pi)}$$ 
be the generating function of $\mathfrak{C}_{n,k}$ with respect to set-valued statistics $\Stat_1, \Stat_2,\dots, \Stat_r$ and integer-valued statistics $\stat_1, \stat_2,\dots, \stat_r$, respectively.
Also let $\mathfrak{C}_{n}:=\mathfrak{C}_{n,0}$ be the set of ordinary connected permutations.
Our investigation is from the following observation, whose proof will be clear later. See Table 3 for initial values.

\begin{table}
\begin{center}
\begin{tabular}{ c|ccccc } 
 $n\backslash k$ & 0&  1 & 2 & 3 &4 \\ \hline
 0 & 1 & ~ & ~ & ~& ~\\
 1 & 3 & 1 & ~ & ~& ~\\
 2 & 13 & 8 & 1 & ~& ~\\
 3 & 71& 62 & 15 & 1& ~\\
 4 & 461 & 516 & 183 & 24& 1\\
\end{tabular}
\end{center}
\caption{Catalan matrix from the seeds $(s_\ell, t_\ell)=(2\ell+3, (\ell+1)^2)$.}
\end{table}

\begin{prop}\label{prop_connected}
Let the seeds be $$(s_\ell, t_\ell)=(2\ell+3, (\ell+1)^2).$$
Then  for $n,k\ge 0$,  we have
\begin{enumerate}
\item $a_{n,k}=|\mathfrak{C}_{n+1,k+1}|.$
\item $a_{n,0}=|\mathfrak{C}_{n+2}|.$
\end{enumerate}
\end{prop}

The highlight in this section is the following, which says that under certain conditions results on connected partial permutations inherit those from partial permutations. Our arguments depend on the partition $\mathfrak{P}_{n,k}=\cup_{r=1}^4 \mathfrak{P}^{(r)}_{n,k}$ as before. For $\pi\in \mathfrak{P}_{n,k}$, we define $\pi'$ to be the submatrix obtained by deleting the last row and column. Recall that $\pi'$ belongs to $\mathfrak{P}_{n-1,k-1}, \mathfrak{P}_{n-1,k}, \mathfrak{P}_{n-1,k}$ or $\mathfrak{P}_{n-1,k-1}$ if $\pi$ belongs to $\mathfrak{P}^{(1)}_{n,k}, \mathfrak{P}^{(2)}_{n,k}, \mathfrak{P}^{(3)}_{n,k}$ or $\mathfrak{P}^{(4)}_{n,k}$ respectively.

\begin{thm}\label{thm_connected}
Let 
$$a_{n,k}=\sum_{\pi\in \mathfrak{P}_{n,k}} \mathbf{Z}^{\mathbf{Stat}(\pi)}:= \sum_{\pi\in \mathfrak{P}_{n,k}}\mathbf{z}_1^{\Stat_1(\pi)}\mathbf{z}_2^{\Stat_2(\pi)}\dots \mathbf{z}_r^{\Stat_r(\pi)}$$
be the entry of the extended Catalan matrix $A^{\mathbf{s}, \mathbf{t}}$ from the extended seeds $\mathbf{s}=(s^{(i)}_\ell)$ and $\mathbf{t}=(t^{(i)}_\ell)$ with respect to set-valued statistics $(\Stat_1, \Stat_2,\dots ,\Stat_r)$. Suppose the following hold:
\begin{enumerate}
\item For any $\rho \in \mathfrak{P}_{n-1,k-1}$, we have $\displaystyle \sum_{\pi\in \mathfrak{P}^{(1)}_{n,k},\, \pi'=\rho }\mathbf{Z}^{\mathbf{Stat}(\pi)}=\mathbf{Z}^{\mathbf{Stat}(\rho )}$,
\item For any $\rho \in \mathfrak{P}_{n-1,k}$, we have $\displaystyle \sum_{\pi\in \mathfrak{P}^{(2)}_{n,k} \cup \mathfrak{P}^{(3)}_{n,k},\, \pi'=\rho }\mathbf{Z}^{\mathbf{Stat}(\pi)}= s^{(n)}_k\mathbf{Z}^{\mathbf{Stat}(\rho )}$,
\item For any $\rho \in \mathfrak{P}_{n-1,k+1}$, we have $\displaystyle \sum_{\pi\in \mathfrak{P}^{(4)}_{n,k},\, \pi'=\rho } \mathbf{Z}^{\mathbf{Stat}(\pi)}=t^{(n)}_{k+1}\mathbf{Z}^{\mathbf{Stat}(\rho )}$.
\end{enumerate}
Let $\mathbf{s}^+=(s^{(i+1)}_{\ell+1})$ and $\mathbf{t}^+=(t^{(i+1)}_{\ell+1})$ and $A^{\mathbf{s}^+,\mathbf{t}^+}$ be the extended Catalan matrix generated by $\mathbf{s}^+,\mathbf{t}^+$.
Then
$$\left[A^{\mathbf{s}^+,\mathbf{t}^+}\right]_{n,k}=C_{n+1,k+1}(\mathbf{Z})$$
and
$$C_{n,0}(\mathbf{Z})=t^{(n)}_1C_{n-1,1}(\mathbf{Z}).$$
\end{thm}
\begin{proof}
It suffices to prove $C_{n+1,k+1}(\mathbf{Z})$ satisfies the recurrence. By definition
$C_{1,1}(\mathbf{z})=a_{1,1}=1.$
Now, it is easy to see that $\pi\in \mathfrak{P}_{n,1}$ is connected if and only if $\pi' \notin \mathfrak{P}_{n-1,0}$ and $\pi'$ is not connected, 
hence
\begin{eqnarray*}
C_{n+1,1}(\mathbf{Z}) &=& \sum_{\pi \in \mathfrak{C}_{n+1,1}}  \mathbf{Z}^{\mathbf{Stat}(\pi)}\\
&=& \sum_{\rho \in \mathfrak{C}_{n,1},  \pi\in \mathfrak{P}^{(2)}_{n+1,1} \cup \mathfrak{P}^{(3)}_{n+1,1}, \pi'=\rho}  \mathbf{Z}^{\mathbf{Stat}(\pi)} +\sum_{\rho \in \mathfrak{C}_{n,2}, \pi\in \mathfrak{P}^{(4)}_{n+1,1}, \pi'=\rho}  \mathbf{Z}^{\mathbf{Stat}(\pi)}\\
&=& s^{(n+1)}_1C_{n,1}(\mathbf{Z})+t^{(n+1)}_2C_{n,2}(\mathbf{Z}).
\end{eqnarray*}

For $k\ge 2$, $\pi\in \mathfrak{P}_{n,k}$ is connected if and only if $\pi'$ is connected, hence
$$C_{n+1,k+1}(\mathbf{Z})=C_{n,k}(\mathbf{Z})+ s^{(n+1)}_{k+1}C_{n,k+1}(\mathbf{Z})+t^{(n+1)}_{k+2}C_{n,k+2}(\mathbf{Z})$$
for $k\geq 1$ by a similar argument.
Therefore, $C_{n+1,k+1}(\mathbf{Z})$ satisfies the same recurrence of $a_{n,k}$ and $\left[A^{\mathbf{s}^+,\mathbf{t}^+}\right]_{n,k}=C_{n+1,k+1}(\mathbf{Z})$. 
For the second desired equation, it is easy to see that $\pi \in \mathfrak{C}_{n,0}$ iff $\pi\in \mathfrak{P}^{(4)}_{n,0}$ and $\pi' \in \mathfrak{C}_{n-1,1}$. Hence the result follows.
\end{proof}

We also have an integer-valued analogue of Theorem \ref{thm_connected}. The proof is almost the same so we omit it here. 
Note that Theorem \ref{thm_connected} does not imply Theorem \ref{thm_connected_int} because some integer-valued statistics like $\inv$ do not have a set-valued version which can be encoded in the extended seeds of an extended Catalan matrix. 
\begin{thm}\label{thm_connected_int}
Let $a_{n,k}$ be the entry of the Catalan matrix $A^{\mathbf{s}, \mathbf{t}}$ from the seeds $\mathbf{s}=(s_0,s_1,\ldots)$ and $\mathbf{t}=(t_1,t_2,\ldots)$ and $$a_{n,k}=\sum_{\pi\in \mathfrak{P}_{n,k}} \mathbf{z}^{\mathbf{stat}(\pi)}:= \sum_{\pi\in \mathfrak{P}_{n,k}}z_1^{\stat_1(\pi)}z_2^{\stat_2(\pi)}\cdots z_r^{\stat_r(\pi)}$$
with respect to statistics $(\stat_1, \stat_2,\dots ,\stat_r)$. Suppose the following hold:
\begin{enumerate}
\item For any $\rho \in \mathfrak{P}_{n-1,k-1}$, we have $\displaystyle \sum_{\pi\in \mathfrak{P}^{(1)}_{n,k},\, \pi'=\rho }\mathbf{z}^{\mathbf{stat}(\pi)}=\mathbf{z}^{\mathbf{stat}(\rho )}$,
\item For any $\rho \in \mathfrak{P}_{n-1,k}$, we have $\displaystyle \sum_{\pi\in \mathfrak{P}^{(2)}_{n,k} \cup \mathfrak{P}^{(3)}_{n,k},\, \pi'=\rho }\mathbf{z}^{\mathbf{stat}(\pi)}= s_k\mathbf{z}^{\mathbf{stat}(\rho )}$,
\item For any $\rho \in \mathfrak{P}_{n-1,k+1}$, we have $\displaystyle \sum_{\pi\in \mathfrak{P}^{(4)}_{n,k},\, \pi'=\rho } \mathbf{z}^{\mathbf{stat}(\pi)}=t_{k+1}\mathbf{z}^{\mathbf{stat}(\rho )}$.
\end{enumerate}
Let $\mathbf{s}^+=(s_1, s_2, \ldots)$ and $\mathbf{t}^+=(t_2, t_3,\ldots)$ and $A^{\mathbf{s}^+,\mathbf{t}^+}$ be the Catalan matrix generated.
Then
$$\left[A^{\mathbf{s}^+,\mathbf{t}^+}\right]_{n,k}=C_{n+1,k+1}(\mathbf{z})$$
and
$$C_{n,0}(\mathbf{z})=t_1C_{n-1,1}(\mathbf{z}).$$
\end{thm}

Proposition~\ref{prop_connected} is now immediate by Theorem~\ref{thm_connected_int}. Recall that in the proofs of the Theorem \ref{thm_par_triple} and Theorem \ref{thm_par_cyc_rlmin} 
we use the same partition of $\mathfrak{P}_{n,k}$ as above and the assumptions of Theorem \ref{thm_connected} or Theorem \ref{thm_connected_int} are also satisfied.
Hence we immediately have following corollaries over $\mathfrak{C}_{n,k}$, 
analogous to Theorem \ref{thm_par_triple}, Theorem \ref{thm_par_cyc_rlmin} and \ref{thm_symmetric} respectively.

\begin{cor}
Let the seed sequences be $s_0=(w+p(1+q))q$ and
$$(s_\ell,t_\ell)=\left(([\ell+1]_{w,q}+p[\ell+2]_q)q^{\ell+1},p[\ell+1]_{w,q}[\ell+1]_qq^{2\ell+1}\right)$$ for $\ell\geq 1$, then 
$a_{n,k}$ is the generating function of $\mathfrak{C}_{n+1,k+1}$ with respect to $\mathsf{rlmin}, \mathsf{inv}$ and $\mathsf{wex}$. Namely,
$$a_{n,k}=\sum_{\pi\in \mathfrak{C}_{n+1,k+1}}w^{\rlmin(\pi)}p^{\wex(\pi)}q^{\inv(\pi)}.$$
Moreover, we have
$$\sum_{\pi\in \mathfrak{C}_{n,0}}w^{\rlmin(\pi)} p^{\wex(\pi)}q^{\inv(\pi)}=wpq\sum_{\pi\in \mathfrak{C}_{n-1,1}}w^{\rlmin(\pi)} p^{\wex(\pi)}q^{\inv(\pi)}.$$
\end{cor}

\begin{cor} Let the extended seed sequences be 
$$(s^{(i)}_\ell, t^{(i)}_\ell)=\left(\ell+a_i+w_i+(\ell+1)p_i, (\ell+b_i)(\ell+w_i)p_i\right).$$ Then 
$a_{n,k}$ of the extended Catalan matrix is the generating function of $\mathfrak{C}_{n+1,k+1}$ with respect to $\mathsf{Fix}, \mathsf{Cyc}_{\ge 2}$, $\mathsf{Rlmip}$ and $\Excl$. Namely,
$$a_{n,k}=\sum_{\pi\in \mathfrak{C}_{n+1,k+1}}\mathbf{a}^{\Fix(\pi)}\mathbf{b}^{\Cyc_{\ge 2}(\pi)}\mathbf{w}^{\Rlmip(\pi)}\mathbf{p}^{\Excl(\pi)}.$$
\end{cor}
\begin{cor} Let the extended seed sequences be 
$$(s^{(i)}_\ell, t^{(i)}_\ell)=\left(\ell+a_i+w_i+(\ell+1)p_i, (\ell+a_i)(\ell+w_i)p_i\right).$$ Then 
$a_{n,k}$ of the extended Catalan matrix is the generating function of $\mathfrak{C}_{n+1,k+1}$ with respect to $ \mathsf{Cyc}$, $\mathsf{Rlmip}$ and $\Excl$. Namely,
$$a_{n,k}=\sum_{\pi\in \mathfrak{C}_{n+1,k+1}}\mathbf{a}^{\Cyc(\pi)}\mathbf{w}^{\Rlmip(\pi)}\mathbf{p}^{\Excl(\pi)}.$$
Moreover, we have
$$\sum_{\pi\in \mathfrak{C}_{n,0}}\mathbf{a}^{\Cyc(\pi)}\mathbf{w}^{\Rlmip(\pi)}\mathbf{p}^{\Excl(\pi)}=a_nw_np_n\sum_{\pi\in \mathfrak{C}_{n-1,1}}\mathbf{a}^{\Cyc(\pi)}\mathbf{w}^{\Rlmip(\pi)}\mathbf{p}^{\Excl(\pi)}.$$ 
\end{cor}

Again, since $a_i$ and  $w_i$ are symmetric we have the following.
\begin{cor}
For any $n,k\geq 0$, we have 
$$(\Cyc,\Rlmip,\Excl)\sim(\Rlmip,\Cyc,\Excl)$$
over $\mathfrak{C}_{n,k}$.
In particular, $\Cyc$ and $\Rlmip$ have a symmetric joint distribution over $\mathfrak{C}_{n,k}$ for $n,k\ge 0$.
\end{cor}

\section{Cycle-up-down partial permutations}

A permutation $\pi\in \mathfrak{S}_n$ is called {\it cycle-up-down} if in each of its cycle $(i_1 i_2 i_3\dots)$ with $i_1$ being the smallest element then the elements of the cycle form a zig-zag pattern $$i_1<i_2>i_3<\cdots .$$
Let $\mathfrak{U}_n$ be the set of cycle-up-down permutations in $\mathfrak{S}_n$.
Deutsch and Elizalde \cite{Deutsch_Elizalde_11} proved that $|\mathfrak{U}_n|$ is the Euler number $E_{n+1}$, which counts the up-down permutations $\pi$
with $\pi_1>\pi_2<\pi_3 \dots$. The first values of $\{E_n\}_{n\ge 0}$
are $1, 1,1,2,5,16,61,272,\dots$

The counting of cycle-up-down permutations by cycles is especially interesting. Let $u_{n,j}$ be the number of cycle-up-down permutations on $[n]$ with $j$ cycles.  
\begin{thm}\label{deutsch} \cite{Deutsch_Elizalde_11}
We have
$$
\sum_{n\geq 0}\left(\sum_{\pi \in \mathfrak{U}_n}q^{\fix(\pi)}t^{\cyc_{\geq 2}(\pi)}\right) \frac{z^n}{n!}=
\frac{e^{(q-t)z}}{(1-\sin z)^{t}}.
$$
In particular, 
$$\sum_{n\geq 0}\left(\sum_{j=1}^n u_{n,j}t^j\right)\frac{z^n}{n!}=\frac{1}{(1-\sin z)^t}.
$$
\end{thm}
Then $u_{n,j}$ is also the number of Andr\'e Permutations (of both types I and II) on $[n+1]$ with $j+1$ right-to-left minimums \cite{Disanto_14}. In graph theory, it is also known that $(-1)^{n-j}u_{n,j}$ 
is the coefficient of $x^j$ of the co-adjoint polynomial of the complete graph of order $n$~\cite{Csikvari_16}. In this section we generalize the concept of cycle-up-down permutations and above theorem to partial permutations.

\subsection{cycle-up-down partial permutations}
We say $\pi\in \mathfrak{P}_{n,k}$ is cycle-up-down if for each full cycle $(a_1a_2\ldots)$, $a_1$ is the smallest element and $a_1<a_2>a_3<\cdots$;
and for each partial cycle $(a_1a_2\ldots X)$ there is $a_1<a_2>a_3<\cdots$. Note that in a partial cycle $a_1$ is not necessarily the smallest element.
Let $\mathfrak{U}_{n,k}$ be the set of cycle-up-down partial permutations in $\mathfrak{P}_k$. 

Our first result reveals that the generating function of $\mathfrak{U}_{n,k}$ with respect to fixed points and longer cycles can be encoded in seeds.

\begin{thm}\label{thm_par_cyc}
Let the seeds  be $$(s_\ell,t_\ell)=\left(\ell+q,\frac{1}{2} \ell(\ell-1+2t) \right),$$ then 
$a_{n,k}$ is the generating function of cycle-up-down partial permutations with respect to
$\fix$ and $\cyc_{\ge 2}$. Namely,
$$a_{n,k}=\sum_{\pi\in \mathfrak{U}_{n,k}}q^{\fix(\pi)}t^{\cyc_{\geq 2}(\pi)}.$$
\end{thm}
\begin{proof}
Let $$b_{n,k}=\sum_{\pi\in \mathfrak{U}_{n,k}}q^{\fix(\pi)} t^{\cyc_{\geq 2}(\pi)}.$$ 
Again it suffices to show that it satisfies the same recurrence as $a_{n,k}$ does. It is clear $b_{0,0}=1$. 
This time we partition $\mathfrak{U}_{n,k}$ into five disjoint sets according to the cycle $\hat c$ containing $n$:
\begin{itemize}
\item $\mathfrak{U}^{(1)}_{n,k}$: $\hat c=(nX)$.
\item $\mathfrak{U}^{(2)}_{n,k}$: $\hat c=(n)$.
\item $\mathfrak{U}^{(3)}_{n,k}$: $\hat c=(a_1a_2\dots a_rX)$, $r\ge 2$, and $a_1$ is the smallest element in $\hat c$.
\item $\mathfrak{U}^{(4)}_{n,k}$: $\hat c=(a_1a_2\dots a_r)$, $r\ge 2$. (Note that $a_1$ is the smallest element in $\hat c$ by definition.)
\item $\mathfrak{U}^{(5)}_{n,k}$: $\hat c=(a_1a_2\dots a_rX)$, $r\ge 2$, and $a_1$ is not the smallest element in $\hat c$.
\end{itemize}
Let $b_{n,k}^{(i)}=\sum_{\pi\in \mathfrak{U}^{(i)}_{n,k}}q^{\fix (\pi)} t^{\cyc_{\geq 2} (\pi)}$. 
In the following we shall prove that
\begin{enumerate}
\item $b_{n,k}^{(1)}= b_{n-1,k-1}$, 
\item $b_{n,k}^{(2)}= q\cdot b_{n-1,k}$, 
\item $b_{n,k}^{(3)}= k\cdot b_{n-1,k}$,
\item $b_{n,k}^{(4)}= (k+1)\cdot t\cdot b_{n-1,k+1}$,
\item $b_{n,k}^{(5)}= {k+1\choose 2} b_{n-1,k+1}.$
\end{enumerate}

The cases (1) and (2) are trivial.  For (3), we present a bijection $\phi$ between the 
set  $$S_{n-1,k}=\{(\pi, c): \pi\in \mathfrak{S}_{n-1,k} \text{ and $c$ is a partial cycle of $\pi$} \}$$
and the set $\mathfrak{U}^{(3)}_{n,k}$. 
We need the concept of \emph{local complement}: for $\alpha_i\in \{\alpha_1,\dots , \alpha_r\}$ with $\alpha_1<\alpha_2 <\dots <\alpha_r$, 
the local complement of $\alpha_i$ with respect to this set is $\alpha_{r+1-i}$. 
Given $(\pi, c) \in S_{n-1,k}$, with the set of entries in $c$ being $\{\alpha_1,\dots , \alpha_r\}$, we perform the following steps and obtain an element in $\mathfrak{U}^{(3)}_{n,k}$. 
\begin{enumerate}
\item[(i)] insert $n$ into the first place of $c$.
\item[(ii)] replace every entry $i$ of this cycle by its local complement with respect to $\{n, \alpha_1,\dots , \alpha_r\}$, while keep all other cycles intact.
\end{enumerate}
For example, for $n=5$ and $c=(143\mathsf{X})$, we obtain $(1534\mathsf{X})$. This is easily seen a bijection and 
$b_{n,k}^{(3)}= k\cdot b_{n-1,k},$ as 
$$|S_{n-1,k}|=k\cdot b_{n-1,k}$$ and the fact that $\phi((\pi,c))$ and $\pi$ have the same full cycles.

The case (4) can be done similarly from the bijection 
$\psi: S_{n-1,k+1} \to \mathfrak{U}^{(4)}_{n,k}$. However, 
after (ii) above we also erase $X$ and result in a full cycle. 
Then one can check that there is one more full cycles in $\psi((\pi,c))$ than $\pi$ and $$|S_{n-1,k+1}|=(k+1)b_{n-1,k+1}.$$ 
Hence $b_{n,k}^{(4)}= (k+1)\cdot t\cdot b_{n-1,k+1}.$

For (5), let $\min(c)$ denote the smallest element of a cycle $c$. We give a bijection $\varphi$ between  the set 
$$T_{n-1,k+1}=\{(\pi, c_1, c_2): \pi \in \mathfrak{S}_{n-1,k+1}, \text{$c_1,c_2$ are partial cycles of $\pi$, and $\min(c_1)>\min(c_2)$} \}$$
and $\mathfrak{U}^{(5)}_{n,k}$.
Given $(\pi, c_1,c_2) \in T_{n-1,k+1}$ with $c_1=(\alpha_1\alpha_2\dots \alpha_r\mathsf{X})$ and $c_2=(\beta_1\beta_2\dots \beta_s\mathsf{X})$ we perform the following steps and obtain an element in $\mathfrak{U}^{(5)}_{n,k}$:
\begin{enumerate}
\item[(i)] Merge $c_1, c_2, n$ into a partial cycle $(\alpha_1\cdots \alpha_r n \beta_1\cdots \beta_{s}\mathsf{X})$. 
\item[(ii)] If $\alpha_{r-1}<\alpha_{r}$, then replace every element of $\beta_i$ and $n$ of this new cycle with its local complement with respect to the set $\{n,\beta_1\cdots \beta_{s}\}$ and leave other cycles intact.  Otherwise we do nothing.
\end{enumerate}
For example, let $n=9$. From $c_1=(273\mathsf{X})$ and $c_2=(16\mathsf{X})$ we obtain $(273916\mathsf{X})$ after two steps; while
from $c_1=(2734\mathsf{X})$ and $c_2=(16\mathsf{X})$ we obtain $(2734916\mathsf{X})$ after the first step, then $(2734196\mathsf{X})$ after the second.
The inverse of this bijection is just to reverse the steps depending on the order of $n$ and the minimum in the resulting partial cycle. Now it is easy to check that
$b_{n,k}^{(5)}={k+1\choose 2}b_{n-1,k+1}$.

As $b_{n,k}=b_{n,k}^{(1)}+b_{n,k}^{(2)}+\dots+b_{n,k}^{(5)}$, the theorem is proved by combining everything above.
\end{proof}

Let $q=t$ we obtain the seeds for the generating function of cycle-up-down permutations enumerated by cycles.
\begin{cor}\label{udcycle}
Let the seeds be $(s_\ell,t_\ell)=\left(\ell+q, \frac{1}{2}\ell(\ell-1+2q)\right)$, then 
$$a_{n,k}=\sum_{\pi\in \mathfrak{U}_{n,k}}q^{\cyc(\pi)}.$$
In particular, $a_{n,0}=\sum_{j=1}^nu_{n,j}q^j$.
\end{cor}

\subsection{Generating function for cycles}
Now we generalize Theorem~\ref{deutsch} to partial permutations.
Let $a_{n,k}=\sum_{\pi\in \mathfrak{U}_{n,k}}q^{\fix(\pi)}t^{\cyc_{\geq 2}(\pi)}$ and
$$\Lambda_k(z):=\sum_{n\geq 0}a_{n,k}\frac{z^n}{n!}$$ be the exponential function of the sequences of the $k$-th column.

\begin{thm}\label{thm_gen_cyc}
For each $k$, we have
$$\Lambda_k(z)= \frac{e^{(q-t)z}}{(1-\sin z)^t}\cdot  \frac{1}{k!} \cdot \left (\frac{\cos z}{1-\sin z}-1\right)^k.$$
\end{thm}

For the proof we need some facts from the theory of \emph{Sheffer matrices}. We summarize what we need in the following. Readers are referred to Chapter 7 of \cite{Aigner_07} for details. 

\medskip
Let $(Q_n)$ be a sequence defined by $Q_0=1$ and $Q_n=q_1q_2\cdots q_n$ for some sequence $(q_n)$. We define a linear map $\Delta$ on the formal power series of $z$ such that $$\Delta(1)=0 \qquad \text{and} \qquad \Delta(z^n)=q_nz^{n-1}.$$ 
For a Catalan matrix $A^{\mathbf{s},\mathbf{t}}=(a_{n,k})$,  let 
$$A_k(z)=\sum_{n\geq 0}a_{n,k} \frac{z^n}{Q_n}$$ 
be the generating function of the $k$-th column $(k\ge 0)$
and denote $B(z):=A_0(z)$. The following result says that under certain conditions we can find a formal power series $F(z)$ such that 
$A_k(z)$ can be determined from $B(z)$ and $F(z)$.

\begin{lem}[\cite{Aigner_07}]\label{lem_F}
Under the above notations, there exists $F(z)$ such that $$A_k(z)=B(z)\frac{F(z)^k}{Q_k}$$ for $k\geq 1$ if and only if 
$$
\begin{cases}
B(z) = aB(z) + bB(z)F(z),\\
\Delta(B(z)F(z)^k) = q_kB(z)F(z)^{k-1} + s_kB(z)F(z)^k + u_{k+1}B(z)F(z)^{k+1}\quad (k\geq 1),
\end{cases}
$$
where $u_k=\frac{t_k}{q_k}$ for $k\geq 0$.
\end{lem}

Let $q_n=n$ for $n\geq 0$ (hence $Q_n=n!$) we have the following.
\begin{cor}[\cite{Aigner_07}]\label{cor_F}
With the assumptions above, there exists $F(z)$ such that $$A_k(z)=\frac{B(z)F(z)^k}{k!}$$ for $k\geq 1$ if and only if 
the $k$-th term of the seed sequences have the form
$$s_k=a+ku \qquad  \text{and} \qquad k(b + (k-1)v)$$ for some $a,b,u,v$. 
Moreover, we have 
$$F'(z) = 1 + uF(z) + vF(z)^2 \qquad \text{and} \qquad B'(z) = aB(z) + bB(z)F(z).$$
\end{cor}

Now we can complete our proof of Theorem \ref{thm_gen_cyc}.\\

\noindent \textit{Proof of Theorem \ref{thm_gen_cyc}.} From Corollary~\ref{udcycle}, the seeds are $(s_k,t_k)=(k+q,k(\frac{1}{2}(k-1)+t))$ (here we let $k$ be the indices to fit the format of Corollary \ref{cor_F}). Now by Corollary \ref{cor_F} we have 
$$\Lambda_k(z)=B(z) \frac{F(z)^k}{k!},\qquad F'(z) = 1 + F(z) + \frac{1}{2}F(z)^2, \quad  \text{and}\quad B(z)' = qB(z) + tB(z)F(z).$$
Solve these differential equations we result in
$$F(z)=\frac{\cos z}{1-\sin z}-1 \qquad \text{and} \qquad  B(z)=\frac{e^{(q-t)z}}{(1-\sin z)^t}.$$
Hence $$\Lambda_k(z)=\frac{e^{(q-t)z}}{(1-\sin z)^t}\cdot  \frac{1}{k!} \cdot \left (\frac{\cos z}{1-\sin z}-1\right)^k,$$ as desired.
\hfill $\Box$

\subsection{Connected cycle-up-down partial permutations}
Along the line of Section 7, we consider connected cycle-up-down partial permutations and show that Theorem \ref{thm_par_cyc} can be extended to connected cycle-up-down partial permuatations in this section.
\medskip

Recall the partition $\mathfrak{U}_{n,k}=\bigcup\limits_{i=1}\limits^{5}\mathfrak{U}^{(i)}_{n,k}$ and the bijections $\phi,\psi,\varphi$ in the proof of Theorem \ref{thm_par_cyc}. For $\rho \in \mathfrak{P}_{n,k}$, let $\mathcal{C}(\rho)$ be the set of partial cycles of $\rho$. From the proof of Theorem \ref{thm_par_cyc} we can easily get the following lemma and proposition and the proofs are omitted.

\begin{lem}\label{lem_par_cyc_con}
The following holds:
\begin{enumerate}
\item[(1)] For $\rho\in\mathfrak{U}_{n-1,k-1}$, we have 
$$q^{\fix(\rho (nX))}t^{\cyc_{\geq 2}(\rho (nX))}=q^{\fix(\rho)}t^{\cyc_{\geq 2}(\rho)}.$$
\item[(2)] For $\rho\in\mathfrak{U}_{n-1,k}$, we have 
$$q^{\fix(\rho (n))}t^{\cyc_{\geq 2}(\rho (n))}=q\cdot q^{\fix(\rho)}t^{\cyc_{\geq 2}(\rho)}.$$
\item[(3)] For $\rho\in\mathfrak{U}_{n-1,k}$, we have 
$$\displaystyle\sum_{c\in \mathcal{C}(\rho)}q^{\fix(\phi(\rho,c))}t^{\cyc_{\geq 2}(\phi(\rho,c))}=k\cdot q^{\fix(\rho)}t^{\cyc_{\geq 2}(\rho)}.$$
\item[(4)] For $\rho\in\mathfrak{U}_{n-1,k+1}$, we have 
$$\displaystyle\sum_{c\in\mathcal{C}(\rho)}q^{\fix(\psi(\rho,c))}t^{\cyc_{\geq 2}(\psi(\rho,c))}=(k+1)\cdot t\cdot q^{\fix(\rho)}t^{\cyc_{\geq 2}(\rho)}.$$
\item[(5)] For $\rho\in\mathfrak{U}_{n-1,k+1}$, we have 
$$\displaystyle\sum_{\text{distinct }c_1,c_2\in\mathcal{C}(\rho)}q^{\fix(\varphi(\rho,c_1,c_2))}t^{\cyc_{\geq 2}(\varphi(\rho,c_1,c_2))}={k+1\choose 2}q^{\fix(\rho)}t^{\cyc_{\geq 2}(\rho)}.$$
\end{enumerate}
\end{lem}

\begin{prop}\label{prop_par_cyc_con}
(i) For $\pi\in\mathfrak{P}_{n,0}$ and $c\in\mathcal{C}(\pi)$, $\pi(n+1\mathsf{X})$ and $\pi(n+1)$ are always not connected. 
(ii) For $n,k\geq 1$ and $\pi\in\mathfrak{P}_{n,k}$, we have
\begin{enumerate}
\item $\pi$ is connected if and only if $\pi(n+1\mathsf{X})$ is connected.
\item $\pi$ is connected if and only if $\pi(n+1)$ is connected.
\item Given $c\in \mathcal{C}(\pi)$, $\pi$ is connected if and only if $\phi(\pi,c)$ is connected.
\item Given $c\in \mathcal{C}(\pi)$, $\pi$ is connected if and only if $\psi(\pi,c)$ is connected.
\item Given distinct $c_1,c_2\in \mathcal{C}(\pi)$, $\pi$ is connected if and only if $\varphi(\pi,c_1,c_2)$ is connected.
\end{enumerate}
\end{prop}

Finally we can obtain the following.  
\begin{thm}\label{thm_par_cyc_con}
Let the seeds be
$$(s_\ell,t_\ell)=(\ell+1+q,\frac{1}{2}(\ell+1)(\ell+2t)).$$
Then $a_{n,k}$ is the generating function of connected cycle-up-down partial permutations with respect to $\fix$ and $\cyc_{\geq2}$. Namely, 
$$a_{n,k}=\sum_{\pi\in\mathfrak{U}_{n,k}\cap\mathfrak{C}_{n,k}}q^{\fix(\pi)}t^{\cyc_{\geq2}(\pi)}.$$
\end{thm}
\begin{proof}
Combining Lemma \ref{lem_par_cyc_con} and Proposition \ref{prop_par_cyc_con} we know that $\sum_{\pi\in\mathfrak{U}_{n,k}\cap\mathfrak{C}_{n,k}}a^{\fix(\pi)}b^{\cyc_{\geq2}(\pi)}$ satisfies the same recurrence as $a_{n,k}$ does. Hence the result follows.
\end{proof}

\section{Concluding remarks}
 In this paper we extend classic statistics (inversion, descent, excedance, weak excedance, right-to-left minimum, cycle, and fixed point) of a permutation to a partial permutation and prove that the generating functions with respect to these statistics can be encoded in the seed sequences of the Catalan matrix. Moreover, it turns out that many classic results on permutation statistics can be carried to partial permutations.  

Unfortunately, we are not able to find a suitable major index $\mathsf{maj}$ of a partial permutation.  A natural checkpoint is its equidistribution with the $\mathsf{inv}$ defined in this paper.

\begin{prob} Find a suitable major index over partial permutations.
\end{prob}

The ordinary cycle-up-down permutations are counted by Euler numbers, which also count many important permutation families, among them the alternating permutations, Andr\'e permutations of both types, and simsun permutations. It will be interesting to have partial versions of these families.

More generally, we may take a combinatorial structure whose counting sequence can be generated as the leftmost column of a Catalan matrix. As shown in this paper, it sounds promising that we might have a definition of the `partial structure' and the enumerators with respect to certain statistics could be encoded in the seeds. We hope that this is a new approach for studying combinatorial statistics. More examples will be given in a forthcoming article.


\end{document}